\documentclass[a4paper,11pt]{article}
\usepackage[utf8]{inputenc}
\usepackage[T1]{fontenc}
\usepackage[english]{babel}
\usepackage{mathpazo}
\usepackage{amsmath,amsthm,amssymb,mathtools,graphicx}
\usepackage[top=2.5cm,bottom=2.5cm,left=2cm,right=2cm]{geometry}
\usepackage[unicode,colorlinks=true,linkcolor=blue,citecolor=red,
urlcolor=blue,pdfencoding=auto,psdextra]{hyperref}
\newtheorem{theorem}{Theorem}[section]
\newtheorem{proposition}[theorem]{Proposition}
\newtheorem{lemma}[theorem]{Lemma}

\theoremstyle{remark}
\newtheorem{remark}[theorem]{Remark}
\newcommand{\E}{\mathbb E}
\renewcommand{\P}{\mathbb P}
\newcommand{\PF}{\mathcal P}
\newcommand{\TV}{d_{\mathrm{TV}}}
\newcommand{\KL}{\mathsf D}

\title{A total variation threshold for the first coordinates of uniform parking functions via Lorentzian polynomials}
\author{Igor Kortchemski\thanks{CNRS \& DMA, \'Ecole normale
sup\'erieure, PSL University,  75005 Paris, France,
\texttt{igor.kortchemski@math.cnrs.fr}}}
\date{}
\hypersetup{pdftitle={A total variation threshold for the first coordinates of uniform parking functions via Lorentzian polynomials},
pdfauthor={Igor Kortchemski}}

\newcommand{\aiuse}[3]{%
  \par\addvspace{0.6\baselineskip}%
  \noindent\textit{#1}\nobreak\hspace{0.5em}%
  \mbox{(#2)}\par\nobreak
  \noindent #3\par
}

\begin{document}
\maketitle

\begin{abstract}
Let $\pi_n$ be a uniform parking function of size $n$. We determine
the scale at which its first $k$ coordinates cease to be close in
total variation to $k$ independent uniform variables on $\{1,\ldots,n\}$.
The distance is comparable to
$1\wedge \sqrt{k\log(k+1)/n}$, uniformly for $n\ge2$ and $1\le k\le n$,
and tends to one above the scale $n/\log n$. At the critical scale
$k\sim\lambda n/\log n$, it converges to
$2\Phi(\sqrt{\lambda/\pi})-1$, where $\Phi$ is the standard normal
distribution function. A key ingredient is the Lorentzian property of the occupation
generating polynomial of parking functions.
Together with a strengthened Han inequality, this yields
a relative entropy estimate for approximation by the product
of the one-coordinate marginals.
\end{abstract}

\section{Introduction}

A uniform parking function of size $n$ has the law of $n$ independent
uniform variables on $\{1,\ldots,n\}$, conditioned to satisfy constraints
on their empirical distribution. The main purpose of this paper is to
determine how many coordinates $k_n$ must be observed before this
conditioning becomes detectable in total variation as $n\to\infty$.
We show in particular that the total variation threshold is $n/\log n$,
whereas Kolmogorov convergence holds whenever $k_n=o(n)$.

A \emph{parking function}
of size $n$ is a vector $x\in[n]^n$ satisfying, for every $1 \leq j \leq n$,
$$
 \#\{i\in[n]:x_i\le j\}\ge j,
$$
where for $n\ge1$ we write $[n]=\{1,\ldots,n\}$. 
Equivalently, its nondecreasing rearrangement satisfies $x_{(j)}\le j$ for every $1 \leq j \leq n$.
Let $\PF_n$ be the set of parking functions; its cardinality is
$(n+1)^{n-1}$ \cite{KW66}. Let $\pi_n$ be uniform on $\PF_n$,
denote by $U_n$ the uniform law on $[n]$, and put
$$
 P_{n,k}=\mathcal L(\pi_n(1),\ldots,\pi_n(k)),\qquad
 Q_{n,k}=U_n^{\otimes k}.
$$

Diaconis and Hicks~\cite{DH17} proved total variation convergence
for one coordinate and studied asymptotic independence of several
coordinates. Bellin \cite{Bel23} proved total variation convergence to independent uniform coordinates when $k_n=o(\sqrt n)$, with bound $O(k_n/\sqrt n)$, and established Kolmogorov convergence when $k_n=o(n)$. The remark following his
Theorem~1 asks whether $k_n=o(n)$ also suffices for total variation.
We answer this question negatively and determine the threshold,
with uniform estimates for all sample sizes.

\begin{theorem}
\label{thm:main}
There exists an absolute constant $C\ge1$ such that, for every
$n\ge2$ and $1\le k\le n$,
\begin{equation}
\label{eq:main}
 \sqrt{\frac{k\log(k+1)}{Cn+k\log(k+1)}}
 \le \TV(P_{n,k},Q_{n,k})
 \le 1 \wedge C\sqrt{\frac{k\log(k+1)}n},
\end{equation}
where $ \TV(P,Q)=\frac12\sum_x|P(x)-Q(x)|$ is the total variation distance.
\end{theorem}

In particular, $\TV(P_{n,k_n},Q_{n,k_n})\to0$ if and only if
$k_n=o(n/\log n)$, whereas the distance tends to one if
$k_n\log n/n\to\infty$.  Thus Kolmogorov convergence may coexist with asymptotic separation
in total variation. The next theorem identifies the transition
between these regimes. Let $\Phi$ be the standard normal distribution
function and set
$$
 F(\lambda)=2\Phi\!\left(\sqrt{\lambda/\pi}\right)-1
 \quad(0\le\lambda<\infty),\qquad F(\infty)=1.
$$

\begin{theorem}
\label{thm:profile}
For every integer sequence $1\le k_n\le n$, if
$k_n\log n/n\to\lambda\in[0,\infty]$, then
$$
 \TV(P_{n,k_n},Q_{n,k_n})\longrightarrow F(\lambda).
$$
\end{theorem}

\begin{figure}[!t]
\centering
\includegraphics[scale=0.5]{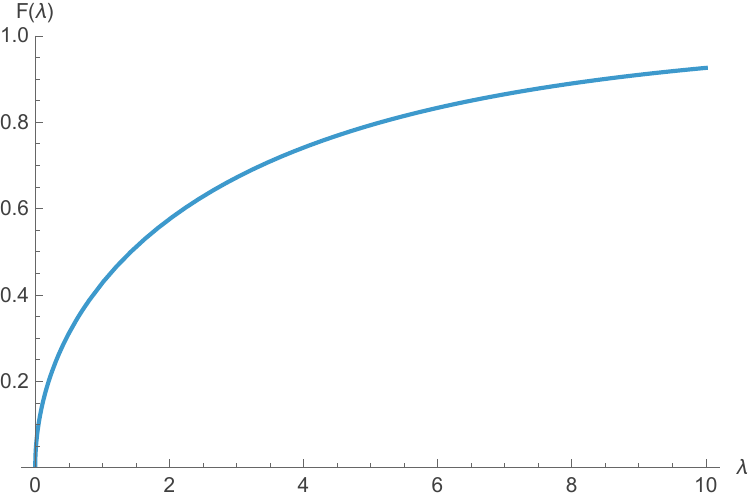}
\caption{The limiting profile
$F(\lambda)=2\Phi(\sqrt{\lambda/\pi})-1$.}\label{fig:F}
\end{figure}

The main input is approximation by independent marginals.
Set $\mu_n=\mathcal L(\pi_n(1))$ and write
$\KL(P\Vert Q)=\sum_xP(x)\log(P(x)/Q(x))$ for relative entropy,
with the usual conventions at zero and natural logarithms;
see~\cite[Chapters~2--4 and~7]{PW25}.

\begin{proposition}
\label{prop:decoupling}
For every $n\ge2$ and $1\le k\le n$,
\begin{equation}
\label{eq:decoupling}
 \KL(P_{n,k}\Vert\mu_n^{\otimes k})
 \le\frac{k(k-1)}{n(n-1)}\log n,
 \qquad
 \TV(P_{n,k},\mu_n^{\otimes k})
 \le\sqrt{\frac{k(k-1)\log n}{2n(n-1)}}.
\end{equation}
\end{proposition}

If $n/\log n=o(k_n)$ and $k_n=o(n/\sqrt{\log n})$,
then $\TV(P_{n,k_n},\mu_n^{\otimes k_n})\to0$ while
$\TV(P_{n,k_n},Q_{n,k_n})\to1$.
At the critical scale, dependence is negligible: the profile is determined
by accumulated marginal bias.

\paragraph{Main steps of the proof.}The support theorem of Br\"and\'en and Huh~\cite{BH20} gives
the Lorentzian property of the occupation polynomial of parking functions.
The strengthened Han inequality of Lee and Seo~\cite{LS26} and
exchangeability then yield the quadratic entropy contraction factor
$k(k-1)/(n(n-1))$ (Lemma~\ref{lem:sample_entropy}), which makes
dependence negligible at the critical scale.  The logarithm reflects the inverse-square-root behavior of the marginal deviations near the endpoints. Indeed, Lemma~\ref{lem:marginal} gives
$$
 \frac1n\sum_{j=1}^n (n\mu_n(j)-1)^2
 =\frac4\pi\frac{\log n}{n}+O(n^{-1}).
$$
At the critical scale, a Gaussian limit for the product log likelihood
yields the profile. The uniform bounds use the product-measure comparison of~\cite{AKVZ26}.

\paragraph{Notation.} Throughout, $c,C$ denote positive absolute constants which may change
from line to line. We write $a\asymp b$ when $ca\le b\le Ca$ uniformly
in the parameters under consideration. For a vector $x$ with
nonnegative integer coordinates, write $|x|=\sum_i x_i$,
$x!=\prod_i x_i!$ and $z^x=\prod_i z_i^{x_i}$.

\paragraph{Acknowledgments.} The author proposed the question of whether $\TV(P_{n,k_n},Q_{n,k_n})\to0$  when $k_{n}=o(n)$ at two open problem workshops, the 2022 Bellairs Probability and Combinatorics workshop and the  2025 ANR Louccoum workshop, and thanks  Omer Angel, \'Etienne Bellin, Peter Keevash as well as Anne-Laure Basdevant, Vanessa Dan, Théo Lenoir for stimulating discussions and attempts to establish this convergence during these meetings.

\paragraph{AI use.} See Section \ref{sec:AI} for a detailed account.

\section{Marginals and occupation counts}
\label{sec:preliminaries}

We denote by $\P_{n}$ the uniform probability measure on $ \mathcal{P}_{n}$ and by $ \mathbb{E}_{n}$ the corresponding expectation. For $\omega \in \mathcal{P}_{n}$ and $1 \leq j \leq n$, set $X_j(\omega)=\#\{i\in[n]:\omega(i)=j\}$. The vector $(X_j)_{1\le j\le n}$ is the occupation vector. Since the law of $\pi_n$ is invariant under permutations
of its coordinates, we have $\E_n[X_j]=n\mu_n(j)$. 

Define
$$
 g_n(j)=n\mu_n(j)-1.
$$
Observe that $ g_n(j)$ represents the deviation of the mean occupation count from the uniform value.

\subsection{The one-coordinate marginal}

\begin{lemma}
\label{lem:marginal}
The sequence $g_n$ is nonincreasing. In addition, there are absolute constants $c,C,\varepsilon>0$ such that the following inequalities hold for every $n\ge2$.  For every $ 1 \leq j \leq n$:
\begin{equation}
\label{eq:g_bounds}
 e^{-1}\le1+g_n(j)\le2,\quad
 |g_n(j)|\le\frac{C}{\sqrt{j\wedge(n+1-j)}},
\end{equation}
and for every $1\le j\le 1 \vee \lfloor\varepsilon n\rfloor$ we have $ g_n(j)\ge{c}/{\sqrt j}$.

Moreover,
\begin{equation}
\label{eq:g_moments}
 \sum_{j=1}^n g_n(j)^2 =\frac4\pi\log n+\mathcal O(1),
 \qquad \sum_{j=1}^n|g_n(j)|^3\le C.
\end{equation}
\end{lemma}

\begin{proof}
Let $q_i=e^{-i}i^{i-1}/i!$ for $i\ge1$ be the Borel
probabilities, so that $\sum_{i\ge1}q_i=1$
(see e.g.~\cite[Section~2.1]{DH17}).
The marginal formula of~\cite[Corollary~1]{DH17} gives, for $1 \leq i <n$,
$$
 \delta_n(i):=g_n(i)-g_n(i+1)=\frac{q_iq_{n+1-i}}{q_{n+1}}>0,
$$
as well as $g_n(1)=(n-1)/(n+1)$ and
$1+g_n(n)=(n/(n+1))^{n-1}=\delta_n(1)$.
These facts give the first bound in~\eqref{eq:g_bounds}.
Also, the symmetry $\delta_n(i)=\delta_n(n+1-i)$ for $2\le i<n$
gives, by telescoping, for $2 \le j \le n$,
\begin{equation}
\label{eq:marginal_reflection}
 g_n(j)+g_n(n+2-j)=g_n(2)+g_n(n)=-\frac2{n+1}.
\end{equation}

Stirling's formula~\cite{Rob55} gives, for $1\le i\le(n+1)/2$,
$$
 q_i=(2\pi)^{-1/2}i^{-3/2}+\mathcal O(i^{-5/2}),\qquad
 \delta_n(i)=q_i\bigl(1+\mathcal O(i/n)\bigr),
$$
where $ \mathcal{O}$ is uniform in
$n\ge2$ and $1\le i\le(n+1)/2$ as $n \rightarrow\infty$.
Thus, uniformly for $1\le j\le(n+2)/2$, telescoping from $g_n(1)$
and using $\sum_{i<j}iq_i\le C\sqrt j$ yields
$$
 g_n(j)=\sum_{i\ge j}q_i+\mathcal O(\sqrt j/n)
 =\sqrt{\frac2\pi}\,j^{-1/2}
       +\mathcal O\bigl(j^{-3/2}+\sqrt j/n\bigr).
$$
Together with~\eqref{eq:marginal_reflection}, this proves the
absolute bound in~\eqref{eq:g_bounds}.
Since $\sum_{i\ge j}q_i\asymp j^{-1/2}$, choosing $\varepsilon$
small enough gives its positive lower bound; $g_n(1)\ge1/3$
covers $\varepsilon n<1$.
Squaring the last display gives
$g_n(j)^2=2/(\pi j)+O(j^{-2}+n^{-1})$ on this half of the range.
Summing and using~\eqref{eq:marginal_reflection} gives
$\sum_{j=1}^n g_n(j)^2=(4/\pi)\log n+O(1)$.
Together with $g_n(1)\ge1/3$, this yields
$\sum_{j=1}^n g_n(j)^2\asymp\log n$ for all $n\ge2$.
Finally, the absolute bound in~\eqref{eq:g_bounds} gives the
cubic estimate.
\end{proof}

\subsection{Log-concavity and product moments}
\label{sec:logconcavity}

Let $f\in\mathbb R[z_1,\ldots,z_d]$ be a polynomial with
nonnegative coefficients. For a multi-index
$\alpha=(\alpha_1,\ldots,\alpha_d)\in\mathbb Z_{\ge0}^d$, write
$
 \partial^\alpha f
 =\partial_{z_1}^{\alpha_1}\cdots
  \partial_{z_d}^{\alpha_d}f$ and  $\partial^0 f=f$.
We say that $f$ is \emph{strongly log-concave} if, for every
$\alpha\in\mathbb Z_{\ge0}^d$, either $\partial^\alpha f\equiv0$
or the function
$
 z\longmapsto\log\bigl(\partial^\alpha f(z)\bigr)
$
is concave on $(0,\infty)^d$.
Since the coefficients of $f$ are nonnegative, every nonzero
derivative $\partial^\alpha f$ is strictly positive on
$(0,\infty)^d$, so its logarithm is well defined.

For homogeneous polynomials, this is equivalent to the Lorentzian
property~\cite[Theorem~2.30]{BH20}. Introduced by Brändén
and Huh \cite{BH20}, Lorentzian polynomials provide a robust framework for  negative-dependence phenomena.

For another log-concavity property of parking functions,
Pappe~\cite{Pap24} proved B\'ona's conjecture on the coefficient
sequence of the sum-of-entries enumerator, using matroid log-concavity.

\begin{lemma}
\label{lem:occupation_logconcavity}
For $n\ge2$, the probability generating polynomial
$G_n(z)=\E_n[\prod_{j=1}^n z_j^{X_j}]$ is strongly log-concave.
Moreover, for every $z\in(0,\infty)^n$ and $1\le k\le n$,
\begin{equation}
\label{eq:product_moments}
 \E_n\!\left[\prod_{r=1}^k z_{\pi_n(r)}\right]
 \le\left(\sum_{j=1}^n\mu_n(j)z_j\right)^k.
\end{equation}
\end{lemma}

\begin{proof}
The product
$
 \prod_{i=1}^n(z_1+\cdots+z_i)
$
is strongly log-concave by Br\"and\'en and Huh~\cite[Corollary~2.32]{BH20},
as a product of linear forms with nonnegative coefficients.
Its support consists exactly of the count vectors of parking functions:
choosing a term $z_{a_i}$ in each factor gives $a_i\le i$, so the
first $j$ choices are at most $j$; conversely, every admissible
multiset has a nondecreasing ordering with $a_i\le i$.
By~\cite[Theorem~3.10 (1) implies (4)]{BH20}, replacing the nonzero coefficient
of each monomial $z^x$ by $1/x!$ gives a strongly log-concave
polynomial.  But this is $G_n$ up to a positive
constant. Indeed, a count
vector $x$ of a parking function has $n!/x!$ orderings, so
\begin{equation}
\label{eq:occupation}
 \P_n(X=x)=\frac{n!}{(n+1)^{n-1}}\frac1{x!}.
\end{equation}

We then apply results of \cite{LS26}.
In the notation of~\cite[Section~3.2]{LS26}, take $A=[n]$
and $\Omega_m=[n]^m$ for every $m\ge1$.
The law of $\pi_n$ is exchangeable and we have just shown that its occupation polynomial
is Lorentzian. Applying~\cite[Lemma~3.4(1)]{LS26} repeatedly
therefore shows that the same holds for its marginal, namely that $
 \E_n\left[\prod_{r=1}^k z_{\pi_n(r)}\right]$ is Lorentzian.
 Now apply~\cite[Equation~(3.6)]{LS26} with $m=k$,
$\nu=P_{n,k}$ and $x_j=z_j$. In their notation, $v_j$ is the
expected number of occurrences of $j$ among the first $k$
coordinates, so exchangeability gives
$$
 v_j=\E_n[\#\{r\in[k]:\pi_n(r)=j\}]=k\mu_n(j).
$$
Consequently, for every $z\in(0,\infty)^n$,
$$
 \E_n\left[\prod_{r=1}^k z_{\pi_n(r)}\right]
 \le\left(\frac1k\sum_{j=1}^n k\mu_n(j)z_j\right)^k
 =\left(\sum_{j=1}^n\mu_n(j)z_j\right)^k,
$$
which is~\eqref{eq:product_moments}.
\end{proof}

\begin{remark}
Up to a positive constant, $G_n$ is the volume polynomial of
the Pitman--Stanley polytope~\cite[Theorems~9 and~11]{PS02}.
Its Lorentzian property therefore also follows
from~\cite[Theorem~4.1]{BH20}.
\end{remark}

\section{Approximation by independent marginals}
\label{sec:approximation}

Recall that $ \KL$ denotes relative entropy. 
\begin{lemma}
\label{lem:sample_entropy}
Let $(Z_1,\ldots,Z_N)$ be an exchangeable vector in $[d]^N$, with
$N\ge2$ and common marginal law $\mu$. Suppose that the generating
polynomial of its occupation vector is strongly log-concave. Then,
for $1\le k\le N$,
$$
 \KL\bigl(\mathcal L(Z_1,\ldots,Z_k)\Vert\mu^{\otimes k}\bigr)
 \le\frac{k(k-1)}{N(N-1)}
       \KL\bigl(\mathcal L(Z_1,\ldots,Z_N)\Vert\mu^{\otimes N}\bigr).
$$
\end{lemma}

\begin{proof}
We write $\mathsf H$ for Shannon entropy, which we recall to be defined by $ \mathsf{H}(P)= - \sum_xP(x)\log{P(x)}$.
For $1\le m\le N$, set
$$
 H_m \coloneqq\mathsf H(Z_1,\ldots,Z_m),\qquad
 D_m
 \coloneqq\KL\bigl(\mathcal L(Z_1,\ldots,Z_m)\Vert\mu^{\otimes m}\bigr).
$$
Since each $Z_i$ has law $\mu$, the definition of relative
entropy gives $D_{m}=mH_1-H_m$.

By repeated application of~\cite[Lemma~3.4(1)]{LS26},
with $A=[d]$ and $\Omega_r=[d]^r$ for every $r\ge1$,
the law of each initial segment $(Z_1,\ldots,Z_m)$
is exchangeable and has a Lorentzian occupation polynomial.
For $2\le m\le N$, the strengthened Han inequality
of~\cite[Lemma~3.7]{LS26}, applied to this law, gives
$$
 \sum_{j=1}^m
 \mathsf H(Z_1,\ldots,Z_{j-1},Z_{j+1},\ldots,Z_m)
 \ge (m-2)H_m+\sum_{j=1}^m\mathsf H(Z_j).
$$
By exchangeability, deleting any one coordinate leaves a vector
with the same law as $(Z_1,\ldots,Z_{m-1})$, and each $Z_j$
has law $\mu$. Thus every summand on the left is $H_{m-1}$,
while every summand in the last sum is $H_1$. Consequently, for every $2 \leq m \leq N$, $
 mH_{m-1}\ge(m-2)H_m+mH_1$.
Equivalently, $mD_{m-1}\le(m-2)D_m$.
Thus $D_m/[m(m-1)]$ is nondecreasing for $2\le m\le N$,
which proves the claim for $k\ge2$.
The case $k=1$ follows from $D_1=0$.
\end{proof}

\begin{proof}[Proof of Proposition~\ref{prop:decoupling}]
Fix $n\ge2$ and $1\le k\le n$.
Apply Lemma~\ref{lem:sample_entropy} with $N=d=n$,
$Z_i=\pi_n(i)$ for $1\le i\le n$, and $\mu=\mu_n$,
using Lemma~\ref{lem:occupation_logconcavity}. We obtain
$$
 \KL(P_{n,k}\Vert\mu_n^{\otimes k})
 \le\frac{k(k-1)}{n(n-1)}
       \KL(P_{n,n}\Vert\mu_n^{\otimes n}).
$$
 Since $\pi_n$ is uniform
on $\PF_n$, we have
$$
 \KL(P_{n,n}\Vert\mu_n^{\otimes n})
 =\log\frac{n^n}{(n+1)^{n-1}}-n\KL(\mu_n\Vert U_n)
 \le\log n.
$$
This proves the relative entropy bound. Pinsker's
inequality~\cite[Theorem~7.10]{PW25} gives the total variation bound.
\end{proof}

\section{Total variation estimates}
\label{sec:estimates}

\subsection{Product measures and uniform bounds}

\begin{lemma}
\label{lem:products}
For all $n\ge2$ and $1\le k\le n$,
$$
 \TV(\mu_n^{\otimes k},Q_{n,k})
 \asymp 1 \wedge \sqrt{\frac{k\log(k+1)}n}.
$$
\end{lemma}

\begin{proof}
Put $b_{n,k}=k\log(k+1)/n$. We apply the scalar characterization
of~\cite[Theorem~1.2 and Section~3.3]{AKVZ26} to $k$ factors,
with one-coordinate laws $\mu_n$ and $U_n$.
The midpoint measure and score in that characterization are
$$
 M(j)=\frac{\mu_n(j)+1/n}{2}=\frac{2+g_n(j)}{2n},
 \qquad
 u(j)=\frac{\mu_n(j)-1/n}{\mu_n(j)+1/n}
     =\frac{g_n(j)}{2+g_n(j)}.
$$
The characterization gives
$\TV(\mu_n^{\otimes k},Q_{n,k})\asymp\min\{1,s\}$
whenever $s>0$ satisfies
$$
 k\sum_{j=1}^n M(j)
  \left( \frac{u(j)^2}{s^2} \wedge \frac{|u(j)|}{s} \right)
 \asymp1.
$$
Since $e^{-1}\le1+g_n(j)\le2$, we have
$M(j)\asymp1/n$ and $|u(j)|\asymp|g_n(j)|$,
uniformly in $n$ and $j$.
Taking $s=\sqrt {b_{n,k}}$ therefore reduces the claim to
$$
 \frac{k}{n}\sum_{j=1}^n
   \frac{g_n(j)^2}{b_{n,k}} \wedge
              \frac{|g_n(j)|}{\sqrt {b_{n,k}}}
 \asymp\frac{k}{n}\sum_{j=1}^n
   \frac1{b_{n,k}j} \wedge \frac1{\sqrt{b_{n,k}j}}
 \asymp1.
$$
The first comparison follows from the upper bound at both endpoints
and the lower bound for $j\le1\vee\lfloor\varepsilon n\rfloor$
in Lemma~\ref{lem:marginal}:
the summand on the right decreases in $j$, so
$1\vee\lfloor\varepsilon n\rfloor\asymp n$
makes the sum over this range comparable to the full sum.
For the last comparison, split at $j=n/(k+1)$.
On the upper part, $b_{n,k}j\ge(\log2)/2$, so the contribution is
comparable to $(k/(nb_{n,k}))\sum_{n/(k+1)\le j\le n}j^{-1}\asymp1$.
The remaining terms contribute at most
$Ck\sqrt{n/(k+1)}/(n\sqrt {b_{n,k}})\le C/\sqrt{\log(k+1)}\le C$.
\end{proof}

\begin{proof}[Proof of Theorem~\ref{thm:main}]
Put $b_{n,k}=k\log(k+1)/n$.

\smallskip
\noindent\emph{Upper bound.}
For $b_{n,k}\ge1$, the upper bound follows from
$\TV(P_{n,k},Q_{n,k})\le1$.
For $0<b_{n,k}<1$, Proposition~\ref{prop:decoupling} gives
\begin{equation}
\label{eq:tv_decoupling}
 \TV(P_{n,k},\mu_n^{\otimes k})
 \le\frac{k}{n}\sqrt{\log n}
 \le Cb_{n,k}\sqrt{\log(e/b_{n,k})}.
\end{equation}
Indeed, $n=k\log(k+1)/b_{n,k}$ implies
$\log n\le2\log(k+1)+\log(1/b_{n,k})$, and
$\log(k+1)\ge\log2$.
Since $b_{n,k}\log(e/b_{n,k})\le1$,
Lemma~\ref{lem:products} and the triangle inequality yield
$
 \TV(P_{n,k},Q_{n,k})\le C\sqrt{b_{n,k}}$.
Together with $\TV(P_{n,k},Q_{n,k})\le1$, this proves
the upper bound in~\eqref{eq:main}.

\smallskip
\noindent\emph{Lower bound.}
Since $b_{n,k}\log(e/b_{n,k})\to0$ as $b_{n,k}\downarrow0$,
\eqref{eq:tv_decoupling}, Lemma~\ref{lem:products} and the
triangle inequality give, for some absolute $b_0\in(0,1)$, for every $0<b_{n,k}\le b_0$, $
 \TV(P_{n,k},Q_{n,k})\ge c\sqrt{b_{n,k}}
$.

We next obtain a bound valid for all $b_{n,k}$.
For $x\in[n]^k$, set $
 h(x)=\prod_{r=1}^k\sqrt{1+g_n(x_r)}$.
For every $x\in[n]^k$, the arithmetic--geometric mean
inequality, applied to $P_{n,k}(x)/h(x)$ and
$Q_{n,k}(x)h(x)$, gives
$$
 \min\{P_{n,k}(x),Q_{n,k}(x)\}
 \le\sqrt{P_{n,k}(x)Q_{n,k}(x)}
 \le\frac12\left(
       \frac{P_{n,k}(x)}{h(x)}+Q_{n,k}(x)h(x)\right).
$$
Summing over $x\in[n]^k$, we obtain
$$
 1-\TV(P_{n,k},Q_{n,k})
 =\sum_{x\in[n]^k}\min\{P_{n,k}(x),Q_{n,k}(x)\}
 \le\frac12\left(
       \E_{P_{n,k}}[h^{-1}]+\E_{Q_{n,k}}[h]\right)
 \le\left(\frac1n\sum_{j=1}^n\sqrt{1+g_n(j)}\right)^k,
$$
where the last inequality uses~\eqref{eq:product_moments}
with $z_j=(1+g_n(j))^{-1/2}$ and the product structure of
$Q_{n,k}=U_n^{\otimes k}$.
Since $\sum_jg_n(j)=0$ and $1+g_n(j)\le2$,
$$
 1-\frac1n\sum_{j=1}^n\sqrt{1+g_n(j)}
 =\frac1{2n}\sum_{j=1}^n
       \bigl(\sqrt{1+g_n(j)}-1\bigr)^2
 \ge\frac{\sum_{j=1}^n g_n(j)^2}{2n(1+\sqrt2)^2}.
$$
Consequently, \eqref{eq:g_moments} yields
\begin{equation}
\label{eq:separation}
 \TV(P_{n,k},Q_{n,k})
 \ge1-\exp\!\left(-c\frac{k\log n}{n}\right).
\end{equation}

Since $\log(k+1)\le2\log n$, this also gives
$\TV(P_{n,k},Q_{n,k})\ge1-e^{-cb_{n,k}}$.
For $b_{n,k}\ge b_0$, we therefore obtain
$$
 \TV(P_{n,k},Q_{n,k})
 \ge1-e^{-cb_{n,k}}
 \ge\frac{cb_{n,k}}{1+cb_{n,k}}
 \ge\sqrt{\frac{b_{n,k}}{C+b_{n,k}}},
$$
where the last inequality holds for
$C\ge2/c+1/(c^2b_0)$.
For $0<b_{n,k}\le b_0$, the same lower bound follows from
$\TV(P_{n,k},Q_{n,k})\ge c\sqrt{b_{n,k}}$, after increasing $C$.
\end{proof}

\subsection{The critical profile}

\begin{proof}[Proof of Theorem~\ref{thm:profile}]
Suppose $0\le\lambda<\infty$. By
Proposition~\ref{prop:decoupling}, since
$k_n\sqrt{\log n}/n\to0$, it suffices to determine
the limit for $\mu_n^{\otimes k_n}$. Under $Q_{n,k_n}=U_n^{\otimes k_n}$, the coordinate variables
$V_1,\ldots,V_{k_n}$ are independent and uniform on $[n]$.
Since $\mu_n(j)=(1+g_n(j))/n$, for every
$x=(x_1,\ldots,x_{k_n})\in[n]^{k_n}$ we have
$$
 \frac{\mathrm d\mu_n^{\otimes k_n}}{\mathrm dQ_{n,k_n}}(x)
 =\frac{\prod_{r=1}^{k_n}\mu_n(x_r)}{n^{-k_n}}
 =\prod_{r=1}^{k_n}(1+g_n(x_r)).
$$
All these factors are positive by~\eqref{eq:g_bounds}.
Taking logarithms, define the log likelihood ratio
$$
 L_{n,k_n}
 =\sum_{r=1}^{k_n}\log(1+g_n(V_r)).
$$
Thus, evaluated at the coordinate variables,
$
 \frac{\mathrm d\mu_n^{\otimes k_n}}{\mathrm dQ_{n,k_n}}
       (V_1,\ldots,V_{k_n})
 =e^{L_{n,k_n}}$. Since $e^{L_{n,k_n}}$ is a probability density with respect
to $Q_{n,k_n}$, we have $\E[e^{L_{n,k_n}}]=1$.
Thus the positive and negative parts of
$e^{L_{n,k_n}}-1$ have equal expectations, and
$$
 \TV(\mu_n^{\otimes k_n},Q_{n,k_n})
 =\frac12\E\bigl[|e^{L_{n,k_n}}-1|\bigr]
 =\E\bigl[(1-e^{L_{n,k_n}})_+\bigr].
$$

 To find the limit of this quantity as $ n \rightarrow\infty$, we proceed as follows. For fixed $u\in\mathbb R$, Taylor's formula, $\sum_jg_n(j)=0$
and~\eqref{eq:g_moments} give
$$
 \E[e^{iuL_{n,k_n}}]
 =\left(\frac1n\sum_{j=1}^n(1+g_n(j))^{iu}\right)^{k_{n}} =\left(1-\frac{2(u^2+iu)}{\pi}\frac{\log n}{n}+\mathcal{O}_{u}(n^{-1})\right)^{k_{n}}
  \quad \mathop{\longrightarrow}_{n \rightarrow \infty} \quad e^{-\sigma^2(u^2+iu)/2},
$$
with $\sigma^2=4\lambda/\pi$. Here the implicit constant in
$\mathcal O_u(n^{-1})$ may depend on $u$ but is independent
of $n$. For each fixed $u$, the cubic Taylor remainder is
uniform over the values of $g_n(j)$ because
$e^{-1}\le1+g_n(j)\le2$.
The limit follows from ${k_{n}}\log n/n\to\lambda$ and ${k_{n}}/n\to0$.
L\'evy's continuity theorem therefore shows that
$L_{n,k_n}$ converges in distribution to $\sigma G-\frac{\sigma^2}{2}$, with  $G \sim N(0,1)$. But the function $x\mapsto(1-e^x)_+$ is bounded and continuous, so
$$
 \TV(\mu_n^{\otimes {k_{n}}},Q_{n,{k_{n}}})
 =\E\bigl[(1-e^{L_{n,k_n}})_+\bigr]
  \quad \mathop{\longrightarrow}_{n \rightarrow \infty} \quad
 \E\bigl[(1-e^{\sigma G-\sigma^2/2})_+\bigr]
 =2\Phi(\sigma/2)-1
 =2\Phi\!\left(\sqrt{\lambda/\pi}\right)-1.
$$

The case $\lambda=\infty$ follows from~\eqref{eq:separation}. This completes the proof.
\end{proof}

\section{AI use}
\label{sec:AI}

\subsection{Development of the paper}
This work was developed with substantial assistance from ChatGPT 6 Pro. Its use included mathematical
exploration, the generation and revision of proof arguments,
literature searches, and LaTeX drafting. The author formulated
the questions and directed the successive revisions.

The following selected exchanges illustrate the main mathematical
developments. The durations are the approximate ``thinking times''
displayed by the interface for the selected responses.

\aiuse
  {(Initial prompt) ``Is it true that when $k_n=o(n)$, the total variation between
  the first $k_n$ positions of a uniform parking function of length
  $n$ and $k_n$ i.i.d.\ uniform random variables on
  $\{1,2,\ldots,n\}$ tends to $0$?''}
  {5 min}
  {The model answered in the negative and identified the scale
  $n/\log n$. At the author's request, it subsequently produced
  a fifteen-page draft, which was reduced to five pages through
  multiple successive revisions.}

\aiuse
  {``What is the most precise estimate you can get for
  $\TV(P_{n,k},Q_{n,k})$?''}
  {14 min}
  {The model proposed a lower bound at the scale of
  Theorem~\ref{thm:main}, together with several upper bounds
  that did not yet match it for slowly growing $k$.}

\aiuse
  {``Get a matching upper bound for slowly growing $k$.''}
  {8 min}
  {The model proposed matching bounds,
  $$
  \TV(P_{n,k},Q_{n,k})
  \asymp
  1 \wedge \sqrt{\frac{k\log(k+1)}{n}},
  $$
  using a negative-dependence theorem for lattice-path matroids
  as an additional ingredient. A subsequent request to incorporate
  these estimates produced a seven-page draft.}

\aiuse
  {``What is your assessment of this paper?''}
  {4 min}
  {The model suggested adding convergence of the total variation
  distance to one above the threshold, observing that the
  necessary estimates were already available.}

\bigskip

 This motivated the author to ask the following.

\aiuse
  {``When $k_n$ is asymptotic to $c n/\log(n)$, what is the limiting
  value of the total variation distance? Can you find it?''}
  {13 min}
  {The model proposed the explicit Gaussian profile in
  Theorem~\ref{thm:profile}. Its argument required a substantially
  stronger approximation by the product of the one-coordinate
  marginals. At the author's request, this extension was
  incorporated into a nine-page draft, which was then reduced and simplified through 
  multiple successive revisions.}

\bigskip

Between these milestones, the author requested repeated audits,
clarifications of individual proof steps, checks of the
applications of cited results, and simplifications. Further
requests for alternative arguments and directly applicable
results from the literature led to the present proof.
The first proof of the entropy approximation used Efron's
monotonicity theorem, Poisson thinning, spectral-gap estimates,
and Fisher information. In the final version, this argument
is replaced by strong log-concavity of the occupation generating
polynomial and the strengthened Han inequality.

The author checked all proofs and all applications of cited
results, edited the manuscript, and takes responsibility for
its mathematical content. The exchange involved approximately
150 user prompts and 40 versions of the manuscript. The author estimates the time personally
devoted to the project at roughly two full working days.

\subsection{Automated formalization in Lean}

The \href{https://github.com/igor-kortchemski/total-variation-threshold-parking-functions}{Lean 4 formalization} of this article was developed
with Codex GPT-6 Astra Ultra, using \href{https://github.com/igor-kortchemski/lean-autoformalization-skills}{skills} adapted
from Scott Armstrong's
\href{https://github.com/scottnarmstrong/LeanAutoformalizationSkills}{LeanAutoformalizationSkills}.
The work proceeded in two stages: an initial formalization of
Theorems~\ref{thm:main} and~\ref{thm:profile} and
Proposition~\ref{prop:decoupling}, followed by a reconstruction
of the manuscript's arguments and a review of their correspondence
with the formal proofs. The author estimates the time spent by GPT-6 on
these two stages at roughly three hours: one hour and forty-five
minutes for the first, and one hour and fifteen minutes for the second. The formalization proceeded automatically, with minimal human intervention.

The resulting development includes the supporting lemmas,
the introductory corollaries, and the separate remark on the
Pitman--Stanley volume polynomial. It follows the Borel-tail
argument for the marginal estimates, the log-concavity and
entropy argument for decoupling, the product comparison for the
uniform bounds, and the characteristic-function argument for
the critical profile. The external results used in the proofs
are proved in the forms needed here, sometimes by different internal
arguments. A correspondence report covers 121 mathematical units
of the manuscript; checks of the dependencies occurring in the
proof terms verify that the specified intermediate results are
actually used.

The proofs have been checked by Lean's kernel using only the standard
axioms of Lean and Mathlib. The three main results were additionally
verified with Comparator, NanoDa and con-ron.

\end{document}